\documentclass[11pt]{article}
\usepackage{amssymb,a4wide,latexsym,makeidx,epsfig,fleqn}
\usepackage{amsthm}
\usepackage{amsmath}
\usepackage{graphicx}
\usepackage{appendix}
\usepackage{subfigure}
\usepackage{hyperref}
\usepackage{float}
\usepackage{enumerate}
\usepackage{arydshln}
\newtheorem{theorem}{Theorem}[section]
\newtheorem{lemma}{Lemma}[section]

\newtheorem{proposition}{Proposition}[section]

\newtheorem{claim}{Claim}[section]

\begin{document}
\textwidth 150mm \textheight 225mm
\title{Linear Tur\'{a}n numbers of acyclic quadruple systems\thanks{Supported by the National Natural Science Foundation of China (No. 12271439) and China Scholarship Council (No. 202206290003).}}
\author{{Lin-Peng Zhang\textsuperscript{a,b,c}, Ligong Wang\textsuperscript{a,b,c,}\footnote{Corresponding author.}}\\
{\small \textsuperscript{a} School of Mathematics and Statistics}\\
{\small Northwestern Polytechnical University, Xi'an, Shaanxi 710129, P.R. China.}\\
{\small \textsuperscript{b} Xi'an-Budapest Joint Research Center for Combinatorics}\\
{\small  Northwestern Polytechnical University, Xi'an, Shaanxi 710129, P.R. China.}\\
{\small \textsuperscript{c} International Joint Research Center on Operations Research, Optimization and Artificial Intelligence}\\
{\small  Xi'an, Shaanxi 710129, P.R. China.}\\
{\small E-mail: lpzhangmath@163.com, lgwangmath@163.com}}
\date{}
\maketitle
\begin{center}
\begin{minipage}{135mm}
\vskip 0.3cm
\begin{center}
{\small {\bf Abstract}}
\end{center}
{\small A hypergraph $H$ is called \emph{linear} if every pair of vertices in $H$ is contained in at most one edge.
Given a family $\mathcal{F}$ of $r$-uniform hypergraphs, an $r$-uniform hypergraph $H$ is called \emph{$\mathcal{F}$-free} if $H$ does not contain
any member of $\mathcal{F}$ as a subhypergraph. The \emph{linear Tur\'{a}n number} $ex^{lin}_{r}
(n,\mathcal{F})$ of $\mathcal{F}$ is the maximum number of edges in an $\mathcal{F}$-free
linear $r$-uniform hypergraphs on $n$ vertices. A linear $r$-uniform hypergraph is called \emph{acyclic} if it can be constructed starting from one single edge
then at each step adding a new edge that intersect the union of the vertices of the previous edges in
at most one vertex. Recently, Gy\'{a}rf\'{a}s, Ruszink\'{o} and S\'{a}rk\"{o}zy [\emph{Linear Tur\'{a}n numbers of
acyclic triple systems, European J. Combin. 99 (2022) 103435.}] initiated the study of the linear Tur\'{a}n numbers of acyclic linear triple systems.
In this paper, we extend their results to linear quadruple systems. Among acyclic linear quadruple systems, we concentrate on small trees, paths and
matchings. For the case of small trees, we find
that for a linear tree $T$, $ex^{lin}_{4}(n,T)$ relates to difficult problems on Steiner system $S(2,4,n)$.
For example, we show that $ex^{lin}_{4}(n,P_4)\le \frac{5n}{4}$ with equality holds if and only if the linear quadruple system is the disjoint union of $S(2,4,16)$. Denote by $E^{+}_4$ the linear tree consisting of three pairwise
disjoint quadruples and a fourth one intersecting all of them. We prove that $12\lfloor\frac{n-4}{9}\rfloor\le ex^{lin}_{4}(n,E^{+}_4)\le \frac{14(n-s)}{9}$, where $s$ is the number of vertices in $G$ with degree at least 8.
Denote by $M_k$ and $P_k$ the set of $k$ pairwise disjoint quadruples and the linear path with $k$ quadruples, respectively.
For the case of paths, we show that $ex^{lin}_{4}(n,P_k)\le 2.5kn$. For the case of matchings, we prove that for fixed $k$ and sufficiently large $n$,
$ex^{lin}_{4}(n,M_k)=g(n,k)$ where $g(n,k)$ denotes the maximum number of quadruples that can intersect $k-1$ vertices in a linear quadruple system on $n$
vertices.
 }
\end{minipage}
\end{center}

\section{Introduction}
We use stadard notation and terminology. A hypergraph $H=(V(H),E(H))$ consists of a set $V(H)$ of
vertices and a set $E(H)$ of edges, where each edge is
a subset of $V(H)$. In particular, if each edge in a hypergraph $H$ is an $r$-element subset of $V(H)
$, then $H$ is an $r$-uniform hypergraph (or $r$-graph for short). When $r=2$, it reduce to a simple
graph. Generally, we call a $3$-graph the triple system where each edge is a triple and a $4$-graph the quadruple system where each edge is a quadruple. For any $v\in V(H)$, the \emph{degree} $d(v)$ of
the vertex $v$ is the number of edges containing $v$. For an edge $e\in E(H)$ and a vertex subset
$S\subseteq V(H)$, $e$ and $S$ are \emph{incident} if the edge $e$ contain at least one vertex of
$S$. For any two edges $e,f\in E(H)$, $e$ and $f$ are \emph{intersecting} (or
say $e$ \emph{intersects} $f$) if $|e\cap f|\ge 1$. For positive integers $k$ and $a\le b$, we use $[k]$ and $[a, b]$ to denote the integer set from 1 to $k$ and the integer set from $a$ to $b$, respectively.

For a hypergraph $H$ and a family $\mathcal{F}$ of hypergraphs, $H$ is called \emph{$\mathcal{F}$-free} if $H$ does not contain any member of $\mathcal{F}$ as a subhypergraph.
The \emph{Tur\'{a}n number} $ex_r(n,\mathcal{F})$ of $\mathcal{F}$ is the maximum number of edges in an $\mathcal{F}$-free $r$-graphs on $n$ vertices. The Tur\'{a}n numbers of hypergraphs have
been studied extensively, we refer the reader to the surveys \cite{Fu,FS,Kee,MV} and the book \cite{GP}. Recently, the study on the Tur\'{a}n numbers of hypergraphs have been extended to the linear
hypergraphs. A hypergraph $H$ is \emph{linear} if every pair of vertices in $H$ is contained in at most one edge. Given a family $\mathcal{F}$ of $r$-graphs, the \emph{linear Tur\'{a}n number} $ex^{lin}_{r}
(n,\mathcal{F})$ of $\mathcal{F}$ is the maximum number of edges in an $\mathcal{F}$-free linear $r$-graphs on $n$ vertices. When $\mathcal{F}=\{F\}$, instead of $ex^{lin}_{r}(n,\{F\})$ we write
$ex^{lin}_{r}(n,F)$. This notion was firstly been proposed by Collier-Cartaino, Graber and Jiang \cite{CC} in 2018. However, the study on linear Tur\'{a}n number can be traced back to the famous $(6,3)$-problem which studied by Brown, Erd\H{o}s and S\'{o}s \cite{BES} in 1973.
The \emph{$(6,3)$-problem} says that what is the maximum number of edges of triple systems not carrying three edges
on six vertices. A Berge cycle Berge-$C_k$ of length $k$ in a hypergraph $H$ is an alternating sequence $v_1e_1v_2e_2\cdots v_ke_k$, where
$v_i\in V(H), e_i\in E(H)$ for $i\in [k]$, $v_i, v_{i+1}\in e_i$ for $i\in [k-1]$ and $v_k, v_1\in e_k$.
 In 1976, by applying the regularity lemma \cite{SZ}, Ruzsa and Szemer\'{e}di \cite{RS}
proved the \emph{triangle removal lemma} which can be phrased as
$$n^{2-\frac{c}{\sqrt{\log{n}}}}\le ex^{lin}_{3}(n,\mbox{Berge-}C_3)=o(n^2)$$
where $c>0$ is a constant.
After that, the linear Tur\'{a}n number of Berge-$C_k$ has been studied extensively, see \cite{EGM, FG, GC, GMV, GS, LV, CT}.

A linear $r$-graph is called \emph{acyclic} if it can be constructed starting from one single edge then at each step adding a new edge that intersect the union of the vertices of the previous edges in
at most one vertex. An acyclic linear $r$-graph is a \emph{linear $r$-tree} if at each step the new edge we add intersect the union of the vertices of the previous edges in exactly one vertex. In particular, a \emph{linear $r$-star} is a linear $r$-tree whose edges are all intersecting in the same one vertex.
A \emph{linear $r$-path} is a linear $r$-tree whose edges are all consecutive.
Denote by $T^{(r)}_k$, $S^{(r)}_k$ and $P^{(r)}_k$ the linear $r$-tree with $k$ edges, the linear $r$-star with $k$ edges and the linear $r$-path with $k$ edges, respectively.
Note that $|V(T^{(r)}_k)|=(r-1)k+1$. For the disconnected case, an \emph{$r$-matching} is a set of pairwise disjoint edges in an $r$-graph.
Denote by $M^{(r)}_{k}$ an $r$-matching with $k$ edges.
Recently, Gy\'{a}rf\'{a}s, Ruszink\'{o} and S\'{a}rk\"{o}zy \cite{GRS} initiated the study of
$ex^{lin}_{3}(n,F)$ for $F\in \{T^{(3)}_k, P^{(3)}_3, B_4, P^{(3)}_4, E_4, P^{(3)}_k, M_{k}^{(3)}\}$, where $E(B_4)=\{\{1,2,3\},\{3,4,5\},\{3,6,7\},\{7,8,9\}\}$, $E(E_4)=\{\{1,2,3\},\{4,5,6\},\{7,8,9\},\{1,4,7\}\}$ and \linebreak $V(B_4)=V(E_4)=[9]$.

Later, Carbonero, Fletchcher, Guo, Gy\'{a}rf\'{a}s, Wang and Yan \cite{CFGGWY} conjectured that $ex^{lin}_{3}(n,\linebreak E_4)\sim \frac{3n}{2}$. Furthermore, Fletchcher \cite{WF} proved that $ex^{lin}_{3}(n,E_4)< \frac{5n}{3}$.
Recently, Tang, Wu, Zhang and Zheng \cite{TWZZ} proved that if $H$ is an $E_4$-free 3-graph on $n$ vertices,
then $|E(H)|\le \frac{3(n-s)}{2}$ where $s$ is the number of vertices in $H$ with degree at least 6.

In this paper, we consider the acyclic linear quadruple systems. Among acyclic linear quadruple systems,
we concentrate on general trees, paths, small trees and matchings. For convenience,
we use $T_k$, $S_k$, $P_k$ and $M_k$ to denote the linear 4-tree with $k$ edges, the linear 4-star with $k$ edges,
the linear 4-path with $k$ edges and the 4-matching with $k$ edges, respectively. In the next three subsections,
we give the results for general trees, paths, small trees and matchings respectively.

\subsection{Results for general linear 4-trees}
\quad We can obtain an upper bound of $ex^{lin}_{4}(n,T_k)$ by analyzing the characterization of linear 4-trees.
\begin{proposition}\label{prop1}
Let $T_k$ be a linear 4-tree with $k>1$ quadruples. Then $ex^{lin}_{4}(n,T_k)\le (3k-5)n$.
\end{proposition}
\begin{proof}
Let $H$ be a $T_k$-free linear quadruple system on $n$ vertices with more than $(3k-5)n$ quadruples. We may assume that $H$ is a minimal counterexample.
Then we have that each vertex of $H$ has degree at least $3k-4$, otherwise we can find a smaller counterexample by deleting a vertex with a smaller degree.
Then by the greedy algorithm we can construct a linear 4-tree $T_k$, adding one quadruple at each step such that the new quadruple intersect
the union of vertices of the previous edges in the required vertex.
\end{proof}

We note that lower bounds for linear Tur\'{a}n numbers of linear 4-trees relate to Steiner systems $S(2,4,n)$.
A \emph{Steiner system} $S(t,k,n)$ is a pair $(V,\mathcal{B})$ where $V$ is an $n$-element
vertex set and $\mathcal{B}$ is a family of $k$-element subsets of $V$ called \emph{blocks}
such that each $t$-element subset of $V$ is contained in exactly one block.
Steiner system with $t=2$ and $k=3$ is called \emph{Steiner triple system}, denoted by $STS(n)$.
For more details of Steiner triple systems see \cite{CR}. Steiner system with $t=3$ and $k=4$ is
called \emph{Steiner quadruple system}, denoted by $SQS(n)$. Here, we consider the Steiner system $S(2,4,n)$ on
$n$ vertices. Hanani \cite{Ha} proved that a Steiner system $S(2,4,n)$ exists if and only if $n\equiv 1,4 \pmod{12}$.
Thus, configurations $S(2,4,13)$ and $S(2,4,16)$ exist (we depict these two configurations in the Appendix of this paper).
For more details of Steiner systems $S(2,4,n)$ see survey \cite{RR}.

We can obtain a natural lower bound for $ex^{lin}_{4}(n,T_k)$ when we can use Steiner systems $S(2,4,3k-2)$ as components and $n$ is divisible by $3k-2$.
\begin{proposition}\label{prop2}
If $(3k-2)|n$ and $3k-2\equiv 1,4 \pmod{12}$, then $ex^{lin}_{4}(n,T_k)\ge \frac{n(k-1)}{4}$. This is sharp when $T_k$ is the linear 4-star $S_k$.
\end{proposition}

\subsection{Results for small linear 4-trees and linear 4-paths}
\quad For $k=2$, it is trivial that $ex^{lin}_{4}(n,P_2)=\lfloor\frac{n}{4}\rfloor$. For $k=3$, we have the
following result.
\begin{proposition}\label{prop3}
$ex^{lin}_{4}(n,P_3)\le n$ with equality if and only if the linear quadruple system is the union of disjoint Steiner systems $S(2,4,13)$.
\end{proposition}

There are four non-isomorphic linear 4-trees $T_k$ with $k=4$. The case of linear 4-star $S_4$ is treated in Proposition \ref{prop2}.
There also are three linear 4-trees with four quadruples except the linear 4-star $S_4$.
Denote by $S^{+}_3$ the linear 4-tree obtained from $S_3$ by appending a quadruple at a vertex of degree one (see Figure \ref{fig2}, where
straight lines with 4 vertices indicate quadruples).
\begin{figure}[htbp]
\centering
\includegraphics[scale=0.50]{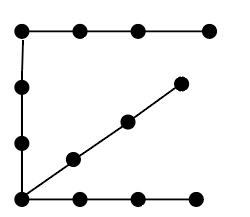}\\
\caption{Configuration $S^{+}_3$}\label{fig2}
\end{figure}
\begin{figure}[htbp]
\centering
\includegraphics[scale=0.37]{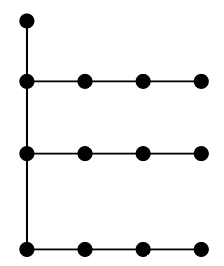}\\
\caption{Configuration $E^{+}_4$}\label{fig3}
\end{figure}

\begin{theorem}\label{th1.1}
Let $F\in \{S^{+}_3,P_4\}$. Then
$$ex^{lin}_{4}(n,F)\le \frac{5n}{4}.$$
Equality holds if and only if the linear quadruple system is the union of disjoint Steiner systems $S(2,4,16)$.
\end{theorem}

Denote by $E^{+}_4$ the linear 4-tree obtained from three pairwise disjoint quadruples by adding one quadruple that intersect all of them (see Figure \ref{fig3}, where straight lines with 4 vertices indicate quadruples). For convenience, we let $\varepsilon=0$ for
$n-4\equiv 0,1,2 \pmod{9}$, $\varepsilon=1$ for $n-4\equiv 3,4 \pmod{9}$, $\varepsilon=2$ for $n-4\equiv 5 \pmod{9}$, $\varepsilon=4$ for $n-4\equiv 6 \pmod{9}$, $\varepsilon=5$ for $n-4\equiv 7 \pmod{9}$, and $\varepsilon=8$ for $n-4\equiv 8 \pmod{9}$.
\begin{theorem}\label{th1.2}
Let $G$ be any $E^+_4$-free linear 4-graph $G$ on $n$ vertices. Then its number of edges satisfies
$$
12\lfloor\frac{n-4}{9}\rfloor+\varepsilon\le |E(G)|\le \frac{14(n-s)}{9},
$$
where $s$ is the number of vertices in $G$ with degree at least 8.
\end{theorem}

For a linear 4-path $P_k$, we slightly improve the general bound of Proposition \ref{prop1}.
\begin{theorem}\label{th1.3}
Let $n,k\ge 1$ be two positive integers. Then $ex^{lin}_{4}(n,P_k)\le 2.5kn$.
\end{theorem}

\subsection{Result for 4-matchings}
\quad In 1959, Erd\H{o}s and Gallai \cite{EG} determined the Tur\'{a}n number of matchings. In 1965, Erd\H{o}s \cite{Er} extended this problem to hypergraph,
proved that for sufficiently large $n$, the maximum number of edges in an $M^{(r)}_{k}$-free $r$-graph occurs if all edges intersect a fixed set of $k-1$ vertices. For linear Tur\'{a}n number, Gy\'{a}rf\'{a}s, Ruszink\'{o} and S\'{a}rk\"{o}zy \cite{GRS} obtained an analogue result.

\begin{theorem}[\cite{GRS}]
Let $n, k$ be two positive integers. For $n>16(k-1)^2+1$,
$$ex^{lin}_{3}(n,M_{k}^{(3)})=f(n,k),$$
where $f(n,k)$ denotes the maximum number of triples that can intersect a fixed $(k-1)$-element vertex subset in a linear triple system on $n$ vertices.
\end{theorem}

In this paper, we extend their result to 4-matchings.

\begin{theorem}\label{th1.4}
Let $n, k$ be two positive integers. For $n>37(k-1)^2+3$,
$$ex^{lin}_{4}(n,M_k)=g(n,k),$$
where $g(n,k)$ denotes the maximum number of quadruples that can intersect a fixed $(k-1)$-element vertex subset in a linear quadruple system on $n$ vertices.
\end{theorem}

Note that Theorem \ref{th1.4} holds only for sufficiently large $n$.
For example, $g(n,2)=\lfloor\frac{n-1}{3}\rfloor<ex^{lin}_{4}(n,M_2)=13$ for $n<40$ as the Steiner system $S(2,4,13)$ shows.
For $k=2$, $n\ge 40$ is the sharp threshold in Theorem \ref{th1.4},
because in a linear quadruple system pairwise intersecting quadruples either form a linear 4-star or form a subsystem of
the Steiner system $S(2,4,13)$.

\section{Proofs of Proposition \ref{prop3}, Theorems \ref{th1.1} and \ref{th1.2}}
\begin{proof}[Proof of Proposition \ref{prop3}]
Firstly, we prove that $ex^{lin}_{4}(n,P_3)\le n$. Let $H$ be a $P_3$-free linear quadruple system on $n$ vertices with more than $n$ quadruples.
We may assume that $H$ is a minimal counterexample. By the minimality, $H$ contains only one connected component. Otherwise $H$ contain at least two connected
components $H_1$ and $H_2$. Then $e(H)=e(H_1)+e(H_2)>n$.
Therefore either $e(H_1)>|V(H_1)|$ or $e(H_2)>|V(H_2)|$ holds, otherwise $e(H)=e(H_1)+e(H_2)\le |V(H_1)|+|V(H_2)|=n$, a contradiction.
Thus, we find a smaller counterexample, a contradiction. We claim that each vertex of $H$ has degree at least two, otherwise we can find a smaller counterexample by deleting a vertex with a smaller degree. We may assume that $H$ contains a vertex with degree at least 5, otherwise $e(H)\le n$. Thus we can select a linear 4-star $S_k$ with center vertex $p$, where $k\ge 5$. Then we select another vertex $q$ in $S_k$. Since the vertex $q$ has degree at least two, there is a quadruple $e$ such that $q\in e$ but $p\notin e$. Then $e$ with two
suitable quadruples of $S_k$ form a $P_3$, a contradiction. Thus, $ex^{lin}_{4}(n,P_3)\le n$.

From the above argument we can also see that if $H$ is a $P_3$-free linear quadruple system on $n$ vertices with exactly $n$ quadruples, then each connected component of $H$ is 4-regular, $i.e.$ each vertex of $H$ has degree 4.
Select one linear 4-star $A=S_4$, we claim that any quadruple intersecting it must be completely inside $A$.
Otherwise, we can find a $P_3$ leading to a contradiction. Thus, each connected component of $H$ is 4-regular on 13 vertices, $i.e.$ the Steiner system $S(2,4,13)$.
\end{proof}

\begin{proof}[Proof of Theorem \ref{th1.1}]
Let $F\in \{S^{+}_3, P_4\}$. Firstly, we prove that $ex^{lin}_{4}(n,F)\le \frac{5n}{4}$.
Let $H$ be an $F$-free linear quadruple system on $n$ vertices with more than $\frac{5n}{4}$ quadruples. We may
assume that $H$ is a minimal counterexample. By the minimality, $H$ contains only one connected component.
We claim that each vertex of $H$ has degree at least two, otherwise we can find a smaller counterexample by deleting a vertex
with a smaller degree. We may assume that $H$ contains a vertex with degree at least 6, otherwise $e(H)\le \frac{5n}{4}$.
Thus we can select a linear 4-star $S_k$ with center vertex $p$, where $k\ge 6$. Let $e_i=\{p,x_i,y_i,z_i\}$, $i\in [k]$ be the quadruples of $S_k$.
In the following, we will first discuss the case of $F=S^{+}_3$ and then $F=P_4$.

{\bf Case 1. $F=S^{+}_3$.} Select another vertex $q\neq p$ in $S_k$. Since the vertex $q$ has degree at least two,
there is a quadruple $e$ such that $q\in e$ but $p\notin e$. Then $e$ with three
suitable quadruples of $S_k$ form an $S^{+}_3$, a contradiction. Therefore if $H$ is an $S^{+}_3$-free linear quadruple system
on $n$ vertices with exactly $\frac{5n}{4}$ quadruples, then each connected component of $H$ is 5-regular.
Select one linear 4-star $A=S_5$, we claim that any quadruple intersecting it must
be completely inside $A$. Otherwise, we can find an $S^{+}_3$ leading to a contradiction. Thus, each connected component of $H$ is 5-regular on 16 vertices, $i.e.$ the Steiner system $S(2,4,16)$.

{\bf Case 2. $F=P_4$.} In this case, we first give the following claim.
\begin{claim}
$V(S_k)=V(H)$.
\end{claim}
\begin{proof}
Indeed, if there exists one vertex $w\in V(H)$ but $w\notin V(S_k)$, then the shortest path $P$ from $w$ to $V(S_k)$ has just one quadruples.
Otherwise, we can extend $P$ to a $P_4$ with two suitable quadruples of $S_k$, a contradiction. Thus there exist two quadruples $f_1, f_2$ containing $w$ such that both of them intersect $S_k$ in vertices different from the center vertex $p$ of $S_k$. In the following, we will discuss two cases.

{\bf Case 1.} $k\ge 7$.
Since $k\ge 7$, we can find a quadruple (say $e_i$) of $S_k$ disjoint from $(f_1\cap S_k)\cup (f_2\cap S_k)$.
Then we can find a $P_4$ containing the quadruples $e_i,e_j,f_1,f_2$ unless
both of $f_1$ and $f_2$ intersect the same three quadruples of $S_k$, where $e_j\in E(S_k)$ is a quadruple containing a vertex from $f_1\cap S_k$.
If both of $f_1$ and $f_2$ intersect the same three quadruples of $S_k$, say $e_1$, $e_2$ and $e_3$,
then we consider an arbitrary vertex $v\in e_i$ with $i\ge 4$.
Since $v$ has degree at least two, there must exist a quadruple $f_3$ containing $v$ different from $e_i$. Note that $w\notin f_3$. Otherwise we can find a  $P_4$ containing the quadruples $f_1,f_3,e_i,e_i'$, where $e_i'\in E(S_k)$ and $e_i'\cap (f_1\cup f_3)=\emptyset$.
Then we claim that either $(i)$:$|f_3\cap (e_1\cup e_2\cup e_3)|=3$ and $f_3\cap (f_1\cup f_2)=\emptyset$ or $(ii)$:$|f_3\cap (f_1\cup f_2)|=2$.
Indeed, if $f_3\cap (e_1\cup e_2\cup e_3)=\emptyset$ then the quadruples $f_3, e_i, e_1, f_1$ form a $P_4$, a contradiction.
If $1\le |f_3\cap (e_1\cup e_2\cup e_3)|\le 2$ and $f_3\cap (f_1\cup f_2)=\emptyset$, then we can find a $P_4$ containing the quadruples $f_1,e_j,e_i,f_3$ ($e_j\in \{e_1,e_2,e_3\}$ satisfies that $e_j\cap f_3=\emptyset$), a contradiction.
If $|f_3\cap (f_1\cup f_2)|=1$ (without loss of generality, $|f_3\cap f_1|=1$),
then we can find a $P_4$ containing the quadruples $e_i,f_3,f_1,f_2$, a contradiction.
Thus, the claim holds. However, there are only 9 pairs of vertices between $f_1$ and $f_2$
and one remaining triple among $e_1$, $e_2$ and $e_3$, but we have at least 12 vertices that may play the role of $v$, a contradiction.

{\bf Case 2.} $k=6$.
We have two possibilities for $f_1$ and $f_2$: $(i)$ both of $f_1$ and $f_2$ intersect the same three quadruples of $S_6$;
$(ii)$ $f_1\cup f_2$ intersect all quadruples of $S_6$. In any other cases,
we can find a $P_4$ defined by $f_1$, $f_2$ and two suitable quadruples leading to a
contradiction. We consider firstly the possibility $(i)$.
Assume that $f_1$ and $f_2$ intersect the same three quadruples $e_1, e_2, e_3$ of $S_6$. Consider an arbitrary vertex $v$
of $e_i$ with $i\ge 4$, there must exist a quadruple $f_3$ containing $v$ different from $e_i$. For another vertex $u$ in $e_i$,
there must exist a quadruple $f_4$ containing $u$ different from $e_i$. Assume $f_3\cap (f_1\cup f_2)=\emptyset$.
If $|f_3\cap (e_1\cup e_2\cup e_3)|=3$, then we can find a $P_4$ defined by $f_4$, one of $f_1,f_2,f_3$ and two suitable quadruples from $S_6$
or $f_4$, two of $f_1,f_2,f_3$ and one suitable quadruple from $S_6$ leading to a contradiction.
If $|f_3\cap (e_1\cup e_2\cup e_3)|\le 2$, then we can find a $P_4$ defined by $f_3$, one of $f_1,f_2$ and two suitable quadruples from $S_6$
leading to a contradiction.
Assume $|f_3\cap (f_1\cup f_2)|=1$. Then we can find a $P_4$ defined by the quadruples $e_i, f_3, f_1, f_2$ leading to a contradiction.
Assume $|f_3\cap (f_1\cup f_2)|=2$. Then we can find a $P_4$ defined by $e_i, f_3, f_1$ and one suitable quadruple from $S_6$ leading to a contradiction.

Consider the second possibility $(ii)$. Without loss of generality, assume that $f_1\cap e_i=\{x_i\}$ and $f_2\cap e_j=\{x_j\}$
for any $i\in \{1,2,3\}$, $j\in \{4,5,6\}$.
Consider another vertex $y_1\in e_1$, there must exist a quadruple $f_3$ containing $y_1$ different from $e_1$.
If $f_3\cap (f_1\cup f_2)=\emptyset$, then we can find
a $P_4$ defined by the quadruples $f_3, e_1, f_1, f_2$ in this order leading to a contradiction. If $|f_3\cap (f_1\cup f_2)|=1$, then we can find a $P_4$ defined by
$f_3, f_1, f_2$ and one suitable quadruple from $S_6$ leading to a contradiction. If $|f_3\cap (f_1\cup f_2)|=2$, then we can find a $P_4$ defined by
$f_3$, one of $f_1, f_2$ and two suitable quadruples from $S_6$ leading to a contradiction.

Combining all the cases, the claim holds, $i.e.$ $V(S_k)=V(H)$.
\end{proof}

Also, we claim that there exists a pair $f_1, f_2$ of intersecting quadruples in $H$ not containing the center vertex $p$ of $S_k$.
Indeed, otherwise $H$ contain at most $\frac{n-1}{4}+\frac{n-1}{3}<\frac{5n}{4}$ edges, a contradiction. We have three possibilities for $f_1$ and $f_2$:
$(i)$ $f_1$ and $f_2$ intersect exactly the same four quadruples of $S_k$, say $e_1, e_2, e_3$ and $e_4$; $(ii)$ $k=7$ and $f_1\cup f_2$
intersect all quadruples of $S_7$; $(iii)$ $k=6$ and $f_1\cup f_2$ intersect all quadruples of $S_6$. In any other cases, we can find a $P_4$
defined by $f_1$, $f_2$ and two suitable quadruples leading to a contradiction. Then for $k>7$, we always have the first possibility $(i)$.
Then we consider the vertex $x_5$ in $e_5$. There must exist a quadruple $f_3$ containing $x_5$ different from $e_5$.
Based on the above argument, $f_3$ cannot intersect $f_1$, but $f_1$ and $f_3$ with two suitable quadruples of $S_k$ form a $P_4$, a contradiction.
Thus we have either $k=7$, $n=22$ and $|E(H)|\ge 28$ or $k=6$, $n=19$ and $|E(H)|\ge 24$.
If $k=7$, then for any pair of intersecting quadruples in $H$ not containing the
center vertex $p$ of $S_7$ we must have the above two possibilities $(i)$ and $(ii)$.
If $k=6$, then for any pair of intersecting quadruples in $H$ not containing the center vertex $p$ of $S_6$
we must have the possibilities $(i)$ and $(iii)$.

We claim that there must exist two disjoint quadruples $g, h$ in $H$ not containing the center vertex $p$ of $S_k$ for $k\in \{6,7\}$.
Indeed, otherwise for $k=7$ ($k=6$) the at least 21 (18) remaining quadruples not containing $p$ are pairwise intersecting. Thus they form a linear 4-star or a subsystem of $S(2,4,13)$. But within 21 (18) vertices there is no room for an $S_{21}$ ($S_{18}$)
and subsystem of an $S(2,4,13)$ cannot have 21 (18) quadruples either. Thus we have $g, h$ as required.
For $k\in \{6, 7\}$, we can find a $P_4$ defined by $g, h$ and two suitable quadruples of $S_k$ leading to a contraction
unless $g$ and $h$ intersect the same four quadruples $e_1, e_2, e_3, e_4$ of $S_k$. Assume that $g=\{x_1, x_2, x_3, x_4\}$ and $h=\{y_1, y_2, y_3, y_4\}$.
For $e_5\in E(S_k)$ ($k\in \{6, 7\}$), there must exist a quadruple $f$ containing $x_5$ but not containing $p$ in $H$.
Furthermore, we have that $f\cap g\neq \emptyset$ and $f\cap h\neq \emptyset$ must
hold for $k\in \{6, 7\}$. Otherwise we can find a $P_4$ formed by $f$, one of $\{g, h\}$ and two suitable quadruples, a contradiction.
If $k=7$, then we can find a $P_4$ defined by $f$, one of $\{g,h\}$ and two suitable quadruples, a contradiction.
If $k=6$, then there must exist two quadruples $f'$ and $f''$ such that $y_5\in f', z_5\in f''$ and $e_5\notin \{f', f''\}$.
Thus we also have $f'\cap g\neq \emptyset, f'\cap h\neq \emptyset$ and $f''\cap g\neq \emptyset, f''\cap h\neq \emptyset$.
Since $f\cap g\neq \emptyset$ and $f\cap h\neq \emptyset$, $f$ and $g$ ($f$ and $h$) must satisfy the possibility $(iii)$.
The same argument hold for $f'$ and $f''$. If $f, f', f''$ are pairwise disjoint, then we can find a $P_4$ defined by two of $\{f, f', f''\}$
and two suitable quadruples of $S_6$, a contradiction. Thus we may assume that $f\cap f'\neq \emptyset$.
If $f\cap f'\in \cup_{i=1}^4 e_i$, then $f$ and $f'$ don't satisfy the possibilities $(i)$ and $(iii)$, a contradiction. Thus, $f\cap f'\in e_6$.

Assume $f$ and $f'$ satisfy the possibility $(i)$. Then $f$ and $f'$ intersect $e_5, e_6$ and the same two quadruples from $\{e_1, e_2, e_3, e_4\}$, say
$e_1, e_2$. If $f''$ is disjoint with $f$ and $f'$, then $f''$ intersect $e_5, e_6, e_3$ and $e_4$. But we can find a $P_4$ defined
by $f''$, one of $\{f, f'\}$ and two suitable quadruples of $S_6$, a contradiction.
We may assume that $f''$ intersect $f$ and $f'$. Note that $f, f'$ and $f''$ intersect $e_6\in E(S_6)$ in the same vertex, say $x_6$.
Consider another vertex $y_6$ in $e_6$, there must exist a quadruple $\tilde{f}$ contains $y_6$ but not containing $p$ in $H$.
Note that $\tilde{f}$ must intersect at least one of $f, f'$ and $f''$.
Then we can find a $P_4$ defined by $\tilde{f}$, two of $\{f,f',f''\}$ and one suitable quadruple of $S_6$ or
$\tilde{f}$, one of $\{f,f',f''\}$ and two suitable quadruples of $S_6$, a contradiction.
This finishes the proof of $ex_{4}^{lin}(n, P_4)\le \frac{5n}{4}$.

Assume $f$ and $f'$ satisfy the possibility $(iii)$. If $f''$ is disjoint with $f$ and $f'$,
then the two statements $f''\cup f$ covers exactly 4 quadruples of $S_6$ and $f''\cup f'$ covers
exactly 4 quadruples of $S_6$ must hold. But this contradict the facts $(f\cap \{e_1, e_2, e_3, e_4\})\cup (f'\cap \{e_1, e_2, e_3, e_4\})=\{e_1, e_2, e_3, e_4\}$ and $(f\cap \{e_1, e_2, e_3, e_4\})\cap (f'\cap \{e_1, e_2, e_3, e_4\})=\emptyset$. We may assume that $f''$ intersect $f$ and $f'$.
Note that $f, f'$ and $f''$ intersect $e_6\in E(S_6)$ in the same vertex, say $x_6$.
Consider another vertex $y_6$ in $e_6$, there must exist a quadruple $\tilde{f}$ contains $y_6$ but not containing $p$ in $H$.
Note that $\tilde{f}$ must intersect at least one of $f, f'$ and $f''$.
Then we can find a $P_4$ defined by $\tilde{f}$, two of $\{f,f',f''\}$ and one suitable quadruple of $S_6$ or
$\tilde{f}$, one of $\{f,f',f''\}$ and two suitable quadruples of $S_6$, a contradiction.

In the case of $|E(H)|=\frac{5n}{4}$, the above argument shows that each connected component of $H$ is 5-regular.
We claim that any connected component of $H$ containing an $S_5$ must contain only the vertices of $S_5$.
Otherwise if $w$ is not on $S_5$, then all the five quadruples $f_1, f_2, f_3, f_4, f_5$
on a vertex $w$ must intersect $S_5$. If there exists one quadruple of $f_1, f_2, f_3, f_4, f_5$ which is disjoint with $S_5$,
then by the connectivity condition we can find a $P_4$, a contradiction.
Moreover, if any $f_i$ intersects $S_5$ in one vertex then it with $f_j$ $(j\neq i)$ and two suitable quadruples of $S_5$
form a $P_4$, a contradiction. Denote by $N_2(w)$ $(N_3(w))$ the set of quadruples from $\{f_1, f_2, f_3, f_4, f_5\}$
intersecting $S_5$ in two (three) vertices. Let $|N_2(w)|=n_2$ and
$|N_3(w)|=n_3$. If $f_i, f_j\in N_3(w)$, then either $f_i$ and $f_j$ intersect the same three quadruples of
$S_5$ or $f_i\cup f_j$ intersect all the five quadruples of $S_5$.
Otherwise we can find a $P_4$ defined by the quadruples $f_i, f_j$ and two suitable quadruples
leading to a contradiction. Then we have $n_3\le 3$. And if $f_i, f_j\in N_2(w)$,
then $f_i$ and $f_j$ must intersect the same two quadruples of $S_5$.
Thus, $2\le n_2\le 3$. It follows from $n_2+n_3=5$ that either $n_2=2$ and $n_3=3$ or $n_2=3$ and $n_3=2$.
If $f_i\in N_2(w)$ and $f_j\in N_3(w)$, then we have that $f_i\cup f_j$ intersect all the five quadruples
of $S_5$. We consider a vertex $v\in f_i)$
with $v\neq w, f_i\in N_2(w)$ and $v\notin V(S_5)$. Note that the above argument for $w$ also holds for $v$.
Thus there is a quadruple $g_i$ containing $v$ intersecting $S_5$ in two quadruples.
Then we can find a $P_4$ defined by $f_i, g_i, f_j$ $(f_j\in N_3(w))$ and one suitable quadruple from $S_5$, a contradiction.
This proves that each connected component of $H$ is 5-regular on 16 vertices, $i.e.$ $S(2,4,16)$.
\end{proof}

Before we prove Theorem \ref{th1.2}, we give the following key lemma.
\begin{lemma}\label{lem1.2.1}
Let $H$ be a $E^+_4$-free graph and $e=\{u,v,w,z\}\in E(H)$ satisfy $D(e)\ge (7,7,6,6)$. Then, the vertex set of all edges sharing a vertex with $\{u,v,w,z\}$,
$$
S=\cup_{f\in E(H), f\cap \{u,v,w,z\}\neq \emptyset}  f,
$$
contains  exactly 22 vertices and all vertices in $S$ have degree at most $7$. The set of edges that contain at least one vertex in $S$,
$$
E_S=\{f:f\in E(G),f\cap S\neq \emptyset\},
$$
contains at most 25 edges, and all elements of $E_S$ are subsets of $S$. In other words, the subgraph $G[S]$ is a connected component of $G$.
\end{lemma}
\begin{proof}
Without loss of generality, assume that $d(z)\ge 7, d(w)\ge 7, d(v)\ge 6$ and $d(u)\ge 6$.  As $D(e)=(7,7,5,2)$ is impossible, we must have $d(w)=d(z)=7$.
Denote by $G(p)$ the set of all vertices distinct from $u,v,w,z$ that lie on the same edge with $p$ for any $p\in \{u,v,w,z\}$.
At first, we note that $G(w)=G(z)$. Otherwise, we assume that there exists an edge $e_1\neq e$ adjacent to $w$ contain some vertex not in $G(z)$. Then
at most two edges adjacent to $z$ other than $e$ contains a vertex in $e_1$, so at least four edges adjacent to $z$ are disjoint from $e_1$. Since $d(x)\ge 5$, we
can take an edge $e_2$ adjacent to $x$ that is disjoint from $e_1$, then take an edge $e_3$ adjacent to $z$ that is disjoint from $e_1$ and $e_2$.
Thus, $e,e_1,e_2,e_3$ forms an $E^+_4$, a contradiction.

Similarly, we have $G(u),G(v)\subset G(w)$. Since the proofs are similar,  it  suffices to show $G(u)\subset G(w)$. Suppose to the contrary that there exists
an edge $e_1\neq e$ adjacent to $u$ contain some vertex not in $G(w)$. Then, we can take an edge $e_3$ adjacent to $z$ that is disjoint from $e_1$. Among the six edges adjacent to $w$, at most three can intersect $e_3$, and at most two can intersect $e_1$. Thus, we can choose $e_2$ adjacent to $w$ that is disjoint from $e_1$ and $e_3$. Thus, $e,e_1,e_2,e_3$ forms an $E^+_4$, a contradiction.

Thus, $S\setminus \{u,v,w,z\}=G(w)=G(z)$ and $G(u), G(v)\subset G(w)$. Define $F$ as the set of all edges in $E(G)$ that adjacent to one of the vertices in $S$, but is disjoint from $\{u,v,w,z\}$.  It suffices to show that $F=\emptyset$.

Denote by $G(z)=\{a_1,b_1,c_1,a_2,b_2,c_2,a_3,b_3,c_3,a_4,b_4,c_4,a_5,b_5,c_5,a_6,b_6,c_6\}$ such that $\{z,a_1,b_1,c_1\},\{z,a_2,b_2,c_2\},\{z,a_3,b_3,c_3\},\{z,a_4,b_4,c_4\},\{z,a_5,b_5,c_5\},\{z,a_6,b_6,c_6\}$ are edges in $E(G)$.

{\bf $(i)$} Define an auxiliary bipartite graph $H=(X_H,Y_H,E_H)$ as follows; $X_H=\{e_i|w\in e_i\}, Y_H=\{e_j|z\in e_j\}, E_H=\{\{e_i,e_j\}|e_i\cap e_j\neq \emptyset\}$.
We claim that $H$ contains a $K_{3,3}$. We first choose $e\in G(u)$. Define $V_1=e\cap S, W_1=\{e_i|e_i\cap V_1\neq \emptyset\}\subset X_H\cup Y_H$.
Hence we have $|V_1|\le 3, |W_1|\le 6, |H-W_1|\ge 6$. Note that if there is no $E^+_4$ in $G$, $H-W_1$ has to be a complete bipartite graph. Since
$|H-W_1|\ge 6$ and two parts have the same order, there must exists a $K_{3,3}$ in $H-W_1$. So $H$ contains a $K_{3,3}$. Thus, we have $H=K_{3,3}\cup K_{3,3}$.

By symmetry we can assume that $\{z,a_1,b_1,c_1\},\{z,a_2,b_2,c_2\},\{z,a_3,b_3,c_3\}$ are in a $K_{3,3}$ and $\{z,a_4,b_4,c_4\},\{z,a_5,b_5,c_5\},\{z,a_6,b_6,c_6\}$ are in the other one. Further, we can assume that $\{w,a_1,a_2,a_3\},\{w,b_1,b_2,b_3\},\{w,c_1,c_2,c_3\},\{w,a_4,a_5,a_6\},\{w,b_4,b_5,b_6\},\{w,c_4,c_5,c_6\}\in E(G)$.

{\bf $(ii)$} Denote by $V_1=\{a_1,b_1,c_1,a_2,b_2,c_2,a_3,b_3,c_3\}$ and $V_2=\{a_4,b_4,c_4,a_5,b_5,c_5,a_6,b_6,c_6\}$. We have symmetry between $V_1$ and $V_2$, and symmetry inside $V_i, i=1,2$ as well. We claim that there exists no edge containing $v$ that contains exactly one vertex from one vertex set in $\{V_1,V_2\}$ and two vertices from another vertex set  in $\{V_1,V_2\}$. Otherwise, we let it be $\{v,a_1,a_4,b_5\}$ by symmetry. Then $\{z,a_1,b_1,c_1\},\{v,a_1,a_4,b_5\},\{z,a_6,b_6,c_6\},\{w,b_1,b_2,b_3\}$ form an $E^+_4$, a contradiction.

{\bf $(iii)$} Let $f\in F$. By symmetry we can assume that $a_1\in f$. Then we have that $a_2,a_3,b_1,c_1\notin f$. We claim that $f$ cannot contain exactly one vertex $a_1$ in $S$.
Otherwise, $\{z,a_1,b_1,c_1\},\{w,b_1,b_2,b_3\},\{z,a_4,b_4,c_4\}, f$ form an $E^+_4$, a contradiction. Then we claim that $b_2,c_2,b_3,c_3\notin f$. Suppose to the contrary that $b_2\in f$. Since $d(v)\ge 6$, there must exists an edge containing $v$ whose other three vertices are all  from $V_1-a_1$, say $e'$. Since at most two edges of $\{z,a_4,b_4,c_4\},\{z,a_5,b_5,c_5\},\{z,a_6,b_6,c_6\}$ intersect $f$, we can assume that $\{z,a_4,b_4,c_4\}\cap f=\emptyset$. Then $\{z,a_1,b_1,c_1\},e',\{z,a_4,b_4,c_4\},f$ form an $E^+_4$, a contradiction.

Therefore by symmetry we can also assume $a_4\in f$. Similarly, we have $b_5,c_5,b_6,c_6\notin f$. So $f$ have exactly two vertices $a_1,a_4$ in $S$. Then there must exists one edge $e''$ containing $v$ whose other three vertices are all from $V_1-a_1$. Therefore $\{z,a_1,b_1,c_1\},f,e'',\{z,a_5,b_5,c_5\}$ form an $E^+_4$, a contradiction.

This completes the proof.
\end{proof}

\begin{proof}[Proof of Theorem \ref{th1.2}]
For the lower bound of $ex^{lin}_{4}(n,E^{+}_4)$, we construct the linear quadruple system as
follows. Note that $E(STS(9))=12$ and there exist four perfect 3-matchings in $STS(9)$, where $STS(9)$ is a Steiner triple system on 9 vertices.
Consider perfect 3-matchings of $m$ disjoint copies of $STS(9)$. We extend each of the four perfect
3-matchings into $3m$ quadruples with four distinct new vertices $a, b, c, d$. This construction is a linear
quadruple system on $9m+4$ vertices with $12m$ quadruples. Note that it is also $E^{+}_4$-free.
To find an $E^{+}_4$ in our construction, we need first select one edge and then select three disjoint edges containing three distinct vertices from our first edge. Without loss of generality, we choose one edge $e$ containing $a$ as our first edge. and then we need choose at least two disjoint edges containing two distinct vertices from $e\setminus \{a\}$, which is impossible since
any two edges form two distinct perfect 3-matchings are intersecting. Thus, our construction is $E^{+}_4$-free.
Adjusting this construction according to divisibility, if $n-4\equiv 0,1,2 \pmod{9}$, then we let $
\varepsilon=0$, $\varepsilon=1$ if $n-4\equiv 3,4 \pmod{9}$, $\varepsilon=2$ if $n-4\equiv 5 \pmod{9}$,
 $\varepsilon=4$ if $n-4\equiv 6 \pmod{9}$, $\varepsilon=5$ if $n-4\equiv 7 \pmod{9}$, $\varepsilon=8$ if $n-4\equiv 8 \pmod{9}$.
 Thus, we have $ex_{4}^{lin}(n,E^{+}_4)\ge 12\lfloor\frac{n-4}{9}\rfloor+\varepsilon$.

Now let us to show the upper bound of $ex^{lin}_{4}(n,E^{+}_4)$. Let $H$ be any linear quadruple system on $n$ vertices.
Denote by $D(e)=\{d(a),d(b),d(c),d(d)\}$ the degree sequence of any edge $e=\{a,b,c,d\}\in E(H)$, where $d(a)\ge d(b)\ge d(c)\ge d(d)$. For any $f=\{u,v,w,z\}\in E(H)$ and $e=\{a,b,c,d\}\in E(H)$, we say $D(f)\ge D(e)$ if $d(u)\ge a, d(v)\ge b, d(w)\ge c$ and $d(z)\ge d$.  Suppose to the contrary that $H$ is the smallest linear $4$-graph of size more than $\frac{14(n-s)}{9}$. Fort any $v\in V(G)$, we define $f(v)=1$ if $d(v)\le 7$ and $f(v)=0$ othewise. We follow the observation from Tang et al. as follows.
$$
\sum_{e\in E(H)} \sum_{v\in V(H),v\in e} \frac{f(v)}{d(v)}=\sum_{v\in V(H)}\sum_{e\in E(H),v\in e} \frac{f(v)}{d(v)}=\sum_{v\in V(H)} f(v)=n-s.
$$
Since $|E(H)|>14(n-s)/9$, we have that there is an edge $e=\{u,v,w,z\}\in E(H)$ satisfying that
$$
\frac{f(u)}{d(u)}+\frac{f(v)}{d(v)}+\frac{f(w)}{d(w)}+\frac{f(z)}{d(z)}<\frac{9}{14}.
$$
Without loss of generality, we assume that $d(u)\le d(v)\le d(w)\le d(z)$. We note that $d(u)\ge 2, d(w)\ge 5$ and $d(z)\ge 7$, as otherwise the above inequality would be violated. Moreover, if $d(z)\ge 8$ then we can find a copy of $E^+_4$ by choosing an edge $e_1\neq e$ adjacent to $u$, an edge $e_2$ adjacent to $w$ that does not share a vertex with $e_1$, and an edge $e_3$ adjacent to $z$ that does not share a vertex with $e_1$ and $e_2$, contradiction. Therefore, $d(z)=7$ and the above inequality implies that $D(e)\ge (7,7,6,6)$.

Assume that $G-S$ be the graph obtained by deleting the vertices $S$ and the edges in $E_S$.  By Lemma \ref{lem1.2.1}, the graph $G-S$ have $n'=n-22$ vertices and at least $|E(G)|-25$ edges. Furthermore, the number of vertices in $G-S$ of degree at least 8 is exactly $s$. Therefore, we have
$$
|E(G-S)|\ge |E(G)|-25>\frac{14(n-s)}{9}-25>\frac{14(n'-s)}{9},
$$
a contradiction. This completes Theorem \ref{th1.2}.
\end{proof}

\section{Proof of Theorem \ref{th1.3}}
Let $H$ be a $P_k$-free linear quadruple system on $n$ vertices with more than $2.5kn$ quadruples.
We may assume that $H$ is a minimal counterexample (neither $k$ nor $n$ can be decreased).
Since for $k=2, 3, 4$ we have sharp results with bounds smaller than $2.5kn$ (see Proposition \ref{prop3} and Theorem \ref{th1.1}), we have that $n>4$ and $k\ge 5$. By the minimality of $H$ for $k$, $H$ contains a path $P=P_{k-1}$ with quadruples $e_i=\{x_{3i-2}, x_{3i-1}, x_{3i}, x_{3i+1}\}$
for $i\in [k-1]$. Also, we claim that each vertex of $H$ has degree at least $2.5k$,
otherwise deleting a vertex with a smaller degree we will get a smaller counterexample, a contradiction.

For convenience, we call the vertices $x_1, x_2, x_3$ the \textit{origin} vertices of $P$, the vertices $x_{3k-4}, x_{3k-3}, x_{3k-2}$ the \textit{terminus} vertices of $P$ and the other vertices of $P$ the \emph{internal} vertices of $P$. We call the vertices of $H$ not in $P$ the \emph{external} vertices.
For $i=1, 2, 3$, denote by $A_1(x_i)$ the set of quadruples in $H$ containing $x_i$, one internal vertex and two external vertices.
For $j= 3k-4, 3k-3, 3k-2$, denote by $B_1(x_j)$ the set of quadruples in $H$ containing $x_j$, one internal vertex and two external vertices.
Note that we have the following inequalities.
\begin{equation}\label{eq1}
|A_1(x_i)|\le 3(k-3), |B_1(x_j)|\le 3(k-3), \mbox{ where } 1\le i\le 3, 3k-4\le j\le 3k-2.
\end{equation}

Let $A_1=A_1(x_1)\cup A_1(x_2)\cup A_1(x_3)$ and $B_1=B_1(x_{3k-4})\cup B_1(x_{3k-3})\cup B_1(x_{3k-2})$.
A \emph{touching pair} is a pair of quadruples $f_1, f_2\in E(H)$ such that $f_1\in A_1$, $f_2\in B_1$ and their internal vertices are the same $x_{3i-1}$ or $x_{3i}$.
\begin{claim}\label{cl1}
There are no touching pairs in $E(H)$.
\end{claim}
\begin{proof}
Suppose to the contrary that $f_1=\{x_1, p, q, x_{3i}\}$ and $f_2=\{x_{3k-2}, s, t, x_{3i}\}$ is a touching pair, where $p, q, s, t$ are distinct.
We can find a $P_k$ defined by $e_{i+1}, e_{i+2}, \cdots, e_{k-1}, f_2, f_1, e_1, \cdots \linebreak , e_{i-1}$, a contradiction.
\end{proof}

Two quadruples $f_1, f_2\in E(H)$ are \textit{crossing} over two consecutive internal vertices $x_i, x_{i+1}$ if $f_1\in A_1, f_2\in B_1$
and $x_{i+1}\in f_1, x_i\in f_2$.
\begin{claim}\label{cl2}
If $f_1, f_2\in E(H)$ are crossing, then $f_1\cap f_2\neq \emptyset$ except for the case $i\equiv 2 \pmod{3}$.
\end{claim}
\begin{proof}
Suppose that $f_1=\{x_1, p, q, x_{i+1}\}$ and $f_2=\{x_{3k-2}, s, t, x_i\}$, where $p, q, s, t$ are distinct.
If $i\equiv 0 \pmod{3}$, then we can find a $P_k$ defined by $f_2, e_{k-1}, e_{k-2}, \cdots, e_{\frac{i}{3}+1} , f_1, e_1, e_2, \cdots, e_{\frac{i}{3}-1}$,
a contradiction. If $i\equiv 1 \pmod{3}$, then we can find a $P_k$ defined by $f_1, e_1, e_2, \cdots, e_{\frac{i-1}{3}}, f_2, e_{k-1}, \linebreak e_{k-2}, \cdots, e_{\frac{i+5}{3}}$,
a contradiction.
\end{proof}

\begin{claim}\label{cl3}
Assume that $f_1, f_2\in E(H)$ are crossing over the internal vertices $x_{3i}, x_{3i+1}$ and $x_a\in f_1$
is the origin vertex, $x_b\in f_2$ is the terminus vertex. Then there exist one original vertex $v\neq x_a$ and one terminus vertex $w\neq x_b$
such that $\{v, x_{3i+1}\}$ is not covered by any quadruple of $A_1$ and $\{w, x_{3i}\}$ is not covered by any quadruple of $B_1$.
Furthermore, if there exist exactly one origin vertex $v\in \{x_1,x_2,x_3\}-x_a$ and one terminus vertex $w\in \{x_{3k-4},x_{3k-3},x_{3k-2}\}-x_b$
such that $\{v, x_{3i+1}\}$ is not covered by any quadruple of $A_1$ and $\{w, x_{3i}\}$ is not covered by any quadruple of $B_1$, then there exists one terminus vertex $u$ such that $\{u, x_{3i-1}\}$ is not covered by any quadruple of $B_1$.
\end{claim}
\begin{proof}
By Claim \ref{cl2}, we have that $f_1=\{x_a, l, s, x_{3i+1}\}, f_2=\{x_b, l, t, x_{3i}\}$ and $l, s, t$ are distinct.
If $\{x_{a'}, x_{3i+1}\}$ and $\{x_{a''}, x_{3i+1}\}$ are covered by two quadruples $g$ and $h$ respectively,
where $\{x_a, x_{a'}, x_{a''}\}=\{x_1,x_2,x_3\}$. Then by Claim \ref{cl2} and the linearity of $H$, we have $|h\cap g|=2$,
a contradiction. Thus, there exists one origin vertex $v\neq x_a$ such that $\{v, x_{3i+1}\}$ is not covered by any quadruple of $A_1$.
The proof of the second statement is similar. Assume that there exist exactly one origin vertex $v\in \{x_1,x_2,x_3\}-x_a$
and one terminus vertex $w\in \{x_{3k-4},x_{3k-3},x_{3k-2}\}-x_b$ such that $\{v, x_{3i+1}\}$ is not covered by
any quadruple of $A_1$ and $\{w, x_{3i}\}$ is not covered by any quadruple of $B_1$. Without loss of generality, we assume
$g_1=\{v, t, p, x_{3i+1}\}\in E(H), g_2=\{w, s, p, x_{3i}\}\in E(H)$, where $l, s, t, p$ are distinct.
Then there must be one terminus vertex $u$ such that $\{u, x_{3i-1}\}$ is not covered by any quadruples of $B_1$.
Otherwise if there exists one terminus vertex $u$ such that $g_3=\{u, a, b, x_{3i-1}\}\in E(H)$, then we have that either $a,b\notin \{l,s\}$
or $a,b\notin \{t,p\}$ holds. Assume that $a,b\notin \{l,s\}$. Then we can find a $P_k$ defined by $g_3, e_{k-1}, \cdots, e_{i+1}, f_1, e_1, \cdots, e_{i-1}$,
a contradiction.
\end{proof}

\begin{claim}
There exist an origin vertex $x_a$ and a terminus vertex $x_b$ of $P$ such that $|A_1(x_a)|+|B_1(x_b)|\le 4(k-3)$.
\end{claim}
\begin{proof}
By the inequality (\ref{eq1}), we have $|A_1|+|B_1|\le 18(k-3)$. Since $|A_1|+|B_1|=|A_1(x_1)+A_1(x_2)+A_1(x_3)+B_1(x_{3k-4})+B_1(x_{3k-3})+B_1(x_{3k-2})|$,
it is enough for us to prove that $|A_1|+|B_1|\le 12(k-3)$. We consider the number of "missing quadruples" from $A_1\cup B_1$ as follows.
For every fixed $i\in \{2, 3, \cdots, k-2\}$, we consider two cases.

{\bf Case 1.} There is no quadruple in $A_1\cup B_1$ containing $x_{3i}$. Then the pairs $\{x_{3i}, x_a\}$ and $\{x_{3i}, x_b\}$ are not covered
by any quadruple of $A_1\cup B_1$ for $x_a\in \{x_1, x_2, x_3\}$ and $x_b\in \{x_{3k-4}, x_{3k-3}, x_{3k-2}\}$. Thus we have six missing
quadruples.

{\bf Case 2.} There is a quadruple $e\in A_1\cup B_1$ containing $x_{3i}$.

{\bf Case 2.1} There is a quadruple $e'\in A_1\cup B_1$ such that $e, e'$ is crossing over $x_{3i}, x^*$. If $e\in A_1$, then we have $x^*=x_{3i-1}$.
By Claim \ref{cl1}, $\{x_{3i}, x_b\}$ is not covered by any quadruple of $B_1$ for any fixed $x_b\in \{x_{3k-4}, x_{3k-3}, x_{3k-2}\}$
and $\{x_{3i-1}, x_a\}$ is not covered by any quadruple of $A_1$ for any fixed $x_a\in \{x_{1}, x_{2}, x_{3}\}$. Thus we have six missing quadruples.
If $e\in B_1$, then $x^*=x_{3i+1}$. By Claim \ref{cl3}, we have at least three missing quadruples.
By Claim \ref{cl1}, $\{x_1, x_{3i}\}$, $\{x_2, x_{3i}\}$ and $\{x_3, x_{3i}\}$ are not covered by any
quadruples of $A_1$. Thus we have at least six missing quadruples.

{\bf Case 2.2.} There exists no $e'\in A_1\cup B_1$ such that $e, e'$ is crossing over $x_{3i}, x^*$. Without loss of generality, assume that $e\in A_1$, then $\{x^*, x_{3k-4}\}$, $\{x^*, x_{3k-3}\}$ and $\{x^*, x_{3k-2}\}$ are not covered by any quadruple of $B_1$. By Claim \ref{cl1}, we have $\{x_{3i}, x_{3k-4}\}$, $\{x_{3i}, x_{3k-3}\}$ and $\{x_{3i}, x_{3k-2}\}$ are not covered by any quadruple of $A_1$. Thus we have at least six missing quadruples.

We conclude that among all cases we have at least six missing quadruples. Thus altogether we have at least $6(k-3)$ missing quadruples in $A_1\cup B_1$. Then $|A_1|+|B_1|\le 12(k-3)$.
\end{proof}

Without loss of generality, assume that $|A_1(x_1)|+|B_1(x_{3k-2})|\le 4(k-3)$.
For $i\in \{2, 3\}$, we use $A_i(x_1)$ and $B_i(x_{3k-2})$ to denote the set of quadruples in $H$
containing $x_1$ and intersecting $P-x_1$ in $i$ vertices and the set of quadruples in $H$
containing $x_{3k-2}$ and intersecting $P-x_{3k-2}$ in $i$ vertices, respectively. Since $H$ is linear, We have
$$3|A_3(x_1)|+2|A_2(x_1)|+|A_1(x_1)|\le 3k-3 \mbox{ , } 3|B_3(x_{3k-2})|+2|B_2(x_{3k-2})|+|B_1(x_{3k-2})|\le 3k-3.$$
Adding above two inequalities to the inequality $|A_1(x_1)|+|B_1(x_{3k-2})|\le 4(k-3)$ we obtain
$$(3|A_3(x_1)|+2|A_2(x_1)|+2|A_1(x_1)|)+(3|B_3(x_{3k-2})|+2|B_2(x_{3k-2})|+2|B_1(x_{3k-2})|)\le 10k-18.$$
Then we have either $3|A_3(x_1)|+2|A_2(x_1)|+2|A_1(x_1)|\le 5k-9$
or $3|B_3(x_{3k-2})|+2|B_2(x_{3k-2})|\linebreak +2|B_1(x_{3k-2})|\le 5k-9$.
It follows that either $\frac{3}{2}|A_3(x_1)|+|A_2(x_1)|+|A_1(x_1)|\le 2.5k-4.5$ or $\frac{3}{2}|B_3(x_{3k-2})|+|B_2(x_{3k-2})|+|B_1(x_{3k-2})|\le 2.5k-4.5$.
Thus we have either $d_H(x_1)=|A_3(x_1)|+|A_2(x_1)|+|A_1(x_1)|\le 2.5k-4.5$ or $d_H(x_{3k-2})=|B_3(x_{3k-2})|+|B_2(x_{3k-2})|+|B_1(x_{3k-2})|\le 2.5k-4.5$,
contradicting the minimum degree condition in a minimal counterexample. \qed

\section{Proof of Theorem \ref{th1.4}}
Before proving Theorem \ref{th1.4}, we first introduce the notion of $t-(m,k,\lambda)$ packing and some useful lemmas.

A $t-(m,k,\lambda)$ \emph{packing} is a pair $(V,\mathcal{B})$ where $V$ is a vertex set on $m$ vertices and $\mathcal{B}$
is a collection of $k$-subsets (called \emph{blocks}) of $V$ such that each $t$-subset of $V$ is contained in \emph{at most} $\lambda$ blocks of $\mathcal{B}$. The \emph{packing number} $D_\lambda(m,k,t)$ is the largest possible number of blocks in a $t-(m,k,\lambda)$ packing.
A $t-(m,k,\lambda)$ packing is called \emph{optimal} if $|\mathcal{B}|=D_\lambda(m,k,t)$.
Thus, an $S(2,4,m)$ is an optimal $2-(m,4,1)$ packing with $D_1(m,4,2)=\frac{m(m-1)}{12}$.
\begin{lemma}[\cite{Br}]
If $m\notin \{8,9,10,11,17,19\}$, then
\begin{description}
\item[$(i)$] for $m\equiv 0,3 \pmod{12}$ an optimal $2-(m,4,1)$ packing is a linear quadruple system whose quadruples cover all pairs of $m$ vertices apart from $m$ pairs which form a copy of $\frac{m}{3}K_3$ (the $\frac{m}{3}$ disjoint union of $K_3$);

\item[$(ii)$] for $m\equiv 2,8 \pmod{12}$ an optimal $2-(m,4,1)$ packing is a linear quadruple system whose quadruples cover all pairs of $m$ vertices apart from $\frac{m}{2}$ pairs which form a copy of $\frac{m}{2} K_2$ (the $\frac{m}{2}$ disjoint union of $K_2$);

\item[$(iii)$] for $m\equiv 5,11 \pmod{12}$ an optimal $2-(m,4,1)$ packing is a linear quadruple system whose quadruples cover all pairs of $m$ vertices apart from $\frac{m+3}{2}$ pairs which form a copy of $K_{1,4}\cup \frac{m-5}{2}K_2$ (the disjoint union of $K_{1,4}$ and $\frac{m-5}{2}$ disjoint edges);

\item[$(iv)$] for $m\equiv 7,10 \pmod{12}$ an optimal $2-(m,4,1)$ packing is a linear quadruple system whose quadruples cover all pairs of $m$ vertices apart from $9$ pairs which form a copy of $K_{3,3}$;

\item[$(v)$] for $m\equiv 6,9 \pmod{12}$ an optimal $2-(m,4,1)$ packing is a linear quadruple system whose quadruples cover all pairs of $m$ vertices apart from $m+3$ pairs which form a copy of $(K_{6}\setminus K_4)\cup \frac{m-6}{3}K_3$ (the disjoint union of $K_{6}\setminus K_4$ and $\frac{m-6}{3}$ disjoint triangles, where $K_{6}\setminus K_4$ denotes the graph obtained from $K_6$ by deleting the edges from some $K_4$).
\end{description}
\end{lemma}

\begin{lemma}[\cite{Br,St}]\label{le4.2}
\begin{equation*}
\left\lfloor\frac{m}{4}\left\lfloor\frac{m-1}{3}\right\rfloor\right\rfloor-D_1(m,4,2)=
\begin{cases}
1, & \mbox{if } m\equiv 7,10 \pmod{12},m\neq 10,19\mbox{ or }m=9,17;\\
2, & \mbox{if } m=8,10,11; \\
3, & \mbox{if } m=19; \\
0, & \mbox{otherwise}.
\end{cases}
\end{equation*}
\end{lemma}
By Lemma \ref{le4.2}, $D_1(8,4,2)=2$, $D_1(9,4,2)=3$, $D_1(10,4,2)=5$, $D_1(11,4,2)=6$, $D_1(17,4,2)=20$ and $D_1(19,4,2)=25$.
An optimal $2-(8,4,1)$ packing $H_1$ can be defined by $(V(H_1), E(H_1))$, where $V(H_1)=[8]$ and $E(H_1)=\{\{1,2,3,4\},\{5,6,7,8\}\}$.
An optimal $2-(9,4,1)$ packing $H_2$ can be defined by $(V(H_2), E(H_2))$, where $V(H_2)=[9]$ and $E(H_2)=\{\{1,2,3,4\},\{1,5,6,7\},\{3,6,8,9\}\}$.
An optimal $2-(10,4,1)$ packing $H_3$ can be defined by $(V(H_3), E(H_3))$, where $V(H_3)=[10]$ and $E(H_3)=\{\{1,2,3,4\},\{2,5,6,7\},\{1,5,9,10\},\{3,7,\linebreak 8,10\},\{4,6,8,9\}\}$. An optimal $2-(11,4,1)$ packing $H_4$ can be defined by $(V(H_4), E(H_4))$, where $V(H_4)=[11]$ and $E(H_4)=\{\{1,2,3,4\},\{1,5,6,7\},
\{1,8,9,10\},\{2,5,8,11\},\{3,6,9,11\},\linebreak \{4,7,10,11\}\}$.
For $m=17$, an optimal $2-(m,4,1)$ packing can be obtained from $S(2,4,16)$ by adding an isolated vertex.
For $m=19$, Stinson \cite{St} gave an optimal $2-(m,4,1)$ packing $H=\{V(H),E(H)\}$, where
$V(H)=[19]$ and $E(H)=\{\{1,2,3,4\},\{1,5,6,10\},\{2,5,7,17\},\{3,6,8,\linebreak 18\},\{4,7,9,18\},\{5,8,9,11\},\{1,7,11,12\},
\{1,8,13, 14\},\{1,9,15,16\},\{2,6,11,15\},\{2,8,12,16\linebreak \},\{3,5,13,19\},\{3,7,14,15\},\{3,9,10,12\},
\{4,5,14,16\},\{4,6,12,19\},\{4,8,15,17\},\{6,7,13,16\linebreak \}, \{6,9,14,17\},\{7,8,10,19\},\{1,17,18,19\},\{2,10,14,18\},\{3,11,16,17\},\{4,10,11,13\},\{5,12,\linebreak 15,18\}\}$.

Then we give a lower bound on $g(n,k)$ for large enough $n$.
\begin{lemma}\label{lem4.1}
$g(n,k)\ge (k-1)\lfloor\frac{n-k+1}{3}\rfloor+\frac{\binom{k-1}{2}}{6}-\frac{7}{2}-\frac{k+2}{6}$ for $n\ge 4k-4$.
\end{lemma}
\begin{proof}
Let $A$ be a fixed $(k-1)$-element subset of $n$ vertices. Then we can find an optimal $2-(n,4,1)$ packing on $A$. By above argument about Steiner system $S(2,4,n)$ and optimal $2-(n,4,1)$
packing, we know this leaves $0, k-1, \frac{k-1}{2}, \frac{k+2}{2}, 9, k+2$ pairs of vertices in $A$ uncovered for
$k-1\notin \{8,9,10,11,17,19\}$. And for $k-1\in \{8,9,10,11,17,19\}$ this leaves $15, 16, 19, 21$ pairs of vertices in $A$ uncovered.
Since $n-k+1\ge 3k-3$, we can extend the vertices of $A$ into quadruples using $k-1$ disjoint perfect 3-matchings of linear triple systems on the $n-k+1$ vertices outside $A$. Thus we have at least
$$(k-1)\left\lfloor\frac{n-k+1}{3}\right\rfloor+\frac{\binom{k-1}{2}-max\{21,k-1,\frac{k-1}{2},\frac{k+2}
{2},k+2\}}{\binom{4}{2}}$$
quadruples, proving the lemma.
\end{proof}

In order to characterize the function $g(n,k)$, we give a simple upper bound on $g(n,k)$ as well.
\begin{lemma}
$g(n,k)\le (k-1)\lfloor\frac{n-k+1}{3}\rfloor+\frac{\binom{k-1}{2}}{2}$.
\end{lemma}
\begin{proof}
Let $A$ be a fixed $(k-1)$-element subset of vertices in a linear quadruple system $H$ on $n$ vertices such that
all quadruples of $H$ intersect $A$. For $j\in \{1,2,3,4\}$, we use $e_j$ to denote the number of edges intersecting $A$ in $j$ vertices.
Note that the quadruples intersecting $A$ in two vertices define a graph with vertex set $A$ and degree sequence $d_i$ for $i\in [k-1]$.
Firstly, by the definition of Steiner triple system $STS(n)$ and Steiner system $S(2,4,n)$, we have $e_3\le \frac{\binom{k-1}{2}}{3}$ and
$e_4\le \frac{\binom{k-1}{2}}{6}$. It follows from $H$ is linear that $e_1\le \sum_{i=1}^{k-1} \frac{n-k+1-2d_i}{3}$ and $e_2=\sum_{i=1}^{k-1} \frac{d_i}{2}$.
Hence, we have
$$e_1+e_2+e_3+e_4\le (k-1)\left\lfloor\frac{n-k+1}{3}\right\rfloor+\frac{\binom{k-1}{2}}{2}-\sum_{i=1}^{k-1} \frac{d_i}{6},$$
proving the lemma.
\end{proof}

We prove Theorem \ref{th1.4} by induction on $k$. For $k=1$, the result is trivial. For $k=2$, the above statement shows that it also holds.
Let $k\ge 3$. Assume that $H$ is an $M_k$-free linear quadruple system on $n$ vertices such that
$|E(H)|>g(n,k)$ and $n>37(k-1)^2+3$. By the inductive hypothesis $H$ contains $M_{k-1}$ with quadruples $X_i=\{a_i,b_i,c_i,d_i\}$ where $i\in [k-1]$.
Similarly, we use $E_2$ to denote the set of quadruples in $H$ intersecting $V(M_{k-1})$ in at least two vertices.
Note that $|E_2|\le \binom{4(k-1)}{2}\}$. When $k=3$, we have $E_2\le 18$ since any quadruple in $E_2$ intersecting $V(M_2)$ in exactly two vertices or four vertices. Since $H$ is $M_k$-free, the set $E_1$ of quadruples in $H$ not in $E_2$ must intersect $V(M_{k-1})$ in exactly one vertex.

Similar with Gy\'{a}rf\'{a}s et al's notion, we call a quadruple $\{a_i,b_i,c_i,d_i\}$ \textit{good} if one of its vertex, say $a_i$, has degree larger than $4(k-1)$ in $E_1$,
otherwise we call it \textit{bad}. Without loss of generality, we assume $M_{k-1}$ contain $j$ good quadruples $X_1, X_2, \cdots, X_j$ with
vertices $a_1, a_2, \cdots, a_j$ of degree larger than $4(k-1)$ in $E_1$, where $0\le j\le k-1$. Note that for a fixed $1\le i\le k-1$, if a vertex in $X_i$
has degree at least four in $E_1$, then the other three vertices of $X_i$ have degree zero in $E_1$. Otherwise we can find an $M_k$ since the quadruple $X_i$
can be replaced by two disjoint quadruples of $E_1$. It follows that the number of quadruples in $E_1$ intersecting a good quadruple $X_i$ equals to the degree
of $a_i$ in $E_1$ and the number of quadruples in $E_1$ intersecting a bad quadruple $X_i$ is at most $4(k-1)$ for $k\ge 4$(12 for $k=3$). Combining all above analysis,
we have
\begin{equation}\label{ineq3}
|E(H)|=|E_2|+|E_1|\le
\begin{cases}
18+j\left\lfloor\frac{n-8}{3}\right\rfloor+12(2-j), &\text{for }k=3;\\
\binom{4(k-1)}{2}+j\left\lfloor\frac{n-4(k-1)}{3}\right\rfloor+4(k-1)(k-1-j), &\text{for }k\ge 4.\\
\end{cases}
\end{equation}
We claim that for $j<k-1$, the inequality \ref{ineq3} contradicts the assumption $E(H)>g(n,k)$.
It is enough for us to check that the right hand side of \ref{ineq3} is smaller than the lower bound of $g(n,k)$
in Lemma \ref{lem4.1}. When $k=3$, it is easy for us to check that above statement hold. When $k\ge 4$, in order to make the above statement hold, by rewriting the second term of \ref{ineq3} as
$j\lfloor\frac{n-k+1}{3}\rfloor-j(k-1)$ and rearranging
we need that
\begin{equation}\label{ineq4}
\binom{4(k-1)}{2}-\frac{k^2-5k-44}{12}-j(k-1)+4(k-1)(k-1-j)<(k-1-j)\left\lfloor\frac{n-k+1}{3}\right\rfloor.
\end{equation}
For \ref{ineq4}, replacing $\lfloor\frac{n-k+1}{3}\rfloor$ by the smaller $\frac{n-k-2}{3}$, rearranging and multiplying by 3 we have
\begin{equation}\label{ineq5}
24(k-1)^2-6(k-1)-\frac{k^2-5k-44}{4}-3j(k-1)+(k-1-j)(13(k-1)+3)<(k-1-j)n.
\end{equation}
The last term on the left hand side of \ref{ineq5} is largest when $j=0$ thus it is enough for us to prove that
\begin{equation}\label{ineq6}
24(k-1)^2-3(k-1)-\frac{k^2-5k-44}{4}-3j(k-1)+13(k-1)^2<(k-1-j)n.
\end{equation}
And since the sum of the three terms with a negative sign on the left hand side of \ref{ineq6} is less than three for $k\ge 4$, we have
$$37(k-1)^2+3<(k-1-j)n,$$
which is true by the assumption $n>37(k-1)^2+3$.

If $j=k-1$, then all quadruples $X_i$ are good for $1\le i\le k-1$. We claim that the vertex subset $A=\{a_1,a_2,\cdots, a_{k-1}\}$
intersect all quadruples in $H$. Otherwise there exists a quadruple $B$ such that $V(B)\cap A=\emptyset$.
Since the degrees of the vertices $a_i$ are larger than $4(k-1)$, by the greedy algorithm we can find $k-1$ pairwise disjoint quadruples
that are disjoint with $B$ as well, a contradiction. We conclude that $A$ intersect all quadruples in $H$ implying that $|E(H)|\le g(n,k)$. Thus, the
theorem holds. \qed

\section*{Appendix}
\appendix
Configurations $S(2,4,13)$ and $S(2,4,16)$ are as follows. $S(2,4,13)=(V(S(2,4,13)), E(S(2,4,\linebreak 13)))$ where $V(S(2,4,13))=[13]$ and
$E(S(2,4,13))=\{\{1,2,3,4\},\{1,5,6,7\},\{1,8,9,10\},\{1,\linebreak 11,12,13\},\{2,5,9,13\},\{2,6,10,11\},\{2,7,8,12\},\{3,5,10,12\},\{3,6,8,13\},\{3,7,9,11\},\{4,5,\linebreak 8,11\},\{4,6,9,12\},\{4,7,10,13\}\}$.
$S(2,4,16)=(V(S(2,4,16)), E(S(2,4,16)))$ where $V(S(2,4,\linebreak 16))=[16]$ and
$E(S(2,4,16))=\{\{1,2,3,4\},\{1,5,6,7\},\{1,8,9,10\},\{1,11,12,13\},\{1,14,15,\linebreak 16\},\{2,5,9 ,13\},\{2,8,12,16\},\{2,11,15,7\},\{2,14,6,10\},\{3,6,8,13\},\{3,9,11,16\},\{3,12,14,\linebreak 7\},\{3,15,5,10\},\{4,7,8,15\},\{4,10,11,6\},
\{4,13,14,9\},\{4,16,5,12\},\{5,8,11,14\},\{6,9,12,\linebreak 15\},\{7,10,13,16\}\}$.

\begin{figure}[htbp]
\centering
\subfigure
{
\begin{minipage}{7cm}
\centering
\includegraphics[scale=0.45]{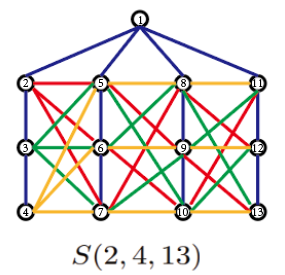}
\end{minipage}
}
\subfigure
{
\begin{minipage}{7cm}
\centering
\includegraphics[scale=0.45]{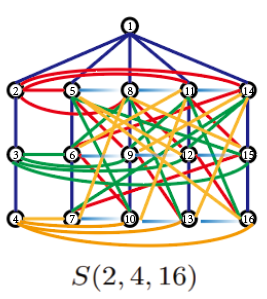}
\end{minipage}
}
\caption{Configurations $S(2,4,13)$ and $S(2,4,16)$}\label{fig1}
\end{figure}

\section*{Declaration of competing interest}
The authors declare that they have no conflict of interest.

\section*{Acknowledgement}
We wish to thank the anonymous referee for his or her valuable suggestions that improved the presentation of this paper.

\end{document}